\documentclass[12pt]{amsart}
\usepackage{amsmath}
\usepackage{amssymb}
\usepackage{amscd}
\usepackage{graphicx}
\usepackage{calrsfs}
\usepackage{enumitem}

\newtheorem{thm}{\sc theorem}

\newtheorem{lem}[thm]{\sc lemma}
\newtheorem{cor}[thm]{\sc corollary}

\newtheorem{trv}[thm]{\sc remark}

\theoremstyle{definition}

\newtheorem{dfn}[thm]{\sc definition}

\theoremstyle{remark}

\newtheorem{exm}[thm]{\sc example}

\newtheorem{qst}[thm]{\sc question}

\def\Z{{\bf Z}}

\def\R{{\bf R}}

\def\S{{\bf S}}

\def\D{{\bf D}}

\begin{document}

\title{Matching Cells}

\maketitle

\begin{center}
Ga\"el
Meigniez

\end{center}

\begin{abstract}
 A (complete) matching of the cells of a triangulated manifold can be thought as
a combinatorial or discrete version of a nonsingular vector field.
 This note gives several
methods for constructing such
matchings.

\end{abstract}

\date{\today}
\section{Introduction} On a
polyhedral complex,
a ``partial matching'' is a family of disjoint pairs of cells such that in each pair,
 one of the two cells
is a hyperface (a face of codimension $1$)
 of the other. Following Forman
 \cite{forman_98}\cite{forman_98_2}\cite{forman_05}\cite{forman_07},
  such objects are regarded as a combinatorial
equivalent to vector fields.
In the literature, most attention has been given to ``discrete Morse theory'',
which concerns partial matchings whith strongly constrained dynamics,
and their relations to the homology of the ambiant complex.
The present note is about \emph{total} matchings involving all the cells,
or whose unmatched cells constitute a prescribed subcomplex; regardless
of their dynamics; we are interested on the \emph{existence} of such objects.
We provide some construction
methods, mainly on triangulated manifolds,
 either allowing oneself to subdivide the triangulation, or not.
The author feels that the methods are more important than the existence results themselves.
A first approach is combinatorial and linear-algebraic,
 making Hall's ``marriage theorem'' play with
cellular homology; two other ones are geometric: a matching is deduced
from an ambient
nonsingular vector field transverse to the cells, or
from a round handle decomposition of the manifold.

 Here are two results.
All manifolds and triangulations are understood smo\-oth ($C^\infty$).

\begin{thm}[Rational homology sphere] 
\label{sphere_thm} Let $M$ be a  rational homology sphere
 of odd dimension.
 
 Then, every triangulation, and more generally
every polyhedral cellulation of $M$ is (totally) matchable.
\end{thm}

\begin{thm}[Matchable subdivision]\label{main_thm}
 Let  $M$
be a compact connected manifold with smooth boundary,
 and let ${\partial_0M}$ be a union of connected components
of $\partial M$ such that $\chi(M,{\partial_0M})=0$. 

Then,
every triangulation of $M$ admits a subdivision matchable rel. $
{\partial_0M}$.

\end{thm}

See below for the meaning of relative matchings. Either of the two sets $
{\partial_0M}$
and $\partial M\setminus{\partial_0M}$ may be empty, or both. In particular:

\begin{cor}\label{closed_cor}
Every closed connected manifold whose Euler characteristic vanishes
admits a (totally) matchable triangulation.
\end{cor}

For manifolds of dimension $3$, this is due to
 E. Gallais \cite{gallais_10}.

\begin{qst}  \emph{ Let $n$ be an odd integer. 
 Is every triangulation of every  closed $n$-manifold (totally)
 matchable?}
\end{qst}
 The question is open even  for $n=3$.
\section{Matching cells, and obstructions to do so}

The rest of this note progressively
 investigates some methods to construct matchings, and also some obstructions;
starting with immediate, elementary remarks; and finishing with the proofs of Theorems 1 and 2.

One denotes by
$\vert A\vert$ the cardinality of the set $A$, by $\D^n\subset\R^n$ the compact unit disk,
 and by $\S^{n-1}:=\partial\D^n$
the $(n-1)$-sphere.

 In a first phase, manifolds are not mandatory, nor simplices.
 Consider generally a polyhedral cellular complex $X$ (the cells are convex polyhedra,
finiteness is understood everywhere)
 and
a subcomplex $Y\subset X$.
 Call two cells {\it incident} to each other if one is a hyperface of the other.
  Write $\Sigma(X,Y)$ (resp.  $\Sigma^n(X,Y)$) (resp.  $\Sigma_0(X,Y)$)
(resp.  $\Sigma_1(X,Y)$) for the set of the cells
(resp.  the $n$-dimensional
 cells) (resp.  the even-dimensional
 cells)  (resp.  the odd-dimensional
 cells) of $X$ 
 not lying in $Y$.

\begin{dfn}
 A \emph{matching} on $X$ \emph{relative to $Y$,}
 or a matching \emph{on the pair $(X,Y)$,}
 is a partition of $\Sigma(X,Y)$
 into incident pairs. \end{dfn}

For $Y=\emptyset$, we write $\Sigma(X)$ (resp. $\Sigma^n(X)$) (resp. $\Sigma_0(X)$)
(resp. $\Sigma_1(X)$) instead of $\Sigma(X,\emptyset)$
(resp. $\Sigma^n(X,\emptyset)$) (resp. $\Sigma_0(X,\emptyset)$)
(resp. $\Sigma_1(X,\emptyset)$); and
we speak of ``a matching on $X$''.

The cases of the complexes of dimension $1$ and of the triangulations of
surfaces
 will easily follow from a few general remarks.
\begin{trv}[Euler characteristic] If $(X,Y)$ is matchable, then
the relative Euler characteristic $\chi(X,Y)=\chi(X)-\chi(Y)$ vanishes.
\end{trv}

Indeed, $\chi(X,Y)=\vert\Sigma_0(X,Y)\vert-\vert\Sigma_1(X,Y)\vert$.

\begin{trv}[Collapse]\label{collapse_trv}  Every collapse of a polyhedral complex $X$ onto 
a subcomplex $Y$ gives a matching on $X$ rel. $Y$.
\end{trv}

 Indeed, a collapse is nothing but a filtration of $X$ by subcomplexes
$(X_n)$, where
 $0\le n\le N$, such that $X_0=Y$ and $X_N=X$; and that $\Sigma(X_n,X_{n-1})$ consists,
for each $1\le n\le N$,
of exactly two incident cells. 

(More precisely, an \emph{orbit}
in a matching is defined as a finite sequence $$\sigma_0,
\sigma_1, \dots\in\Sigma(X,Y)$$
 such that for every odd $k$, the cells
 $\sigma_{k-1}$ and $\sigma_{k+1}$ are two distinct hyperfaces of $\sigma_{k}$,
 and
  $\sigma_{k-1}$ is the mate of  $\sigma_{k}$.
A collapse of $X$ onto $Y$ amounts to a matching
of $X$ rel. $Y$
{without \emph{cyclic} orbit.})

\begin{trv}[Top-dimensional cycle]\label{cycle_trv}

  Every cellulation of the circle admits exactly two matchings.

More generally,
 let $X$ be a polyhedral cellulation of a manifold; let $\ell$
be a simple loop in the $1$-skeleton of the dual cellulation; let $Y\subset X$
be the union of the cells of $X$ disjoint from $\ell$. Then, $X$
admits exactly two matchings rel. $Y$.
\end{trv}

This is obvious.

\begin{exm}[Graphs] \emph{Every connected graph whose Euler
characteristic vanishes is matchable.}
\end{exm}

Indeed, such a graph collapses onto a
 circle.

\begin{exm}[Surfaces] \emph{Let $M$ be a compact, connected $2$-manifold
such that $\chi(M)=0$. 
 Then,  every polyhedral cellulation
 $X$ of $M$ is matchable absolutely, and
relatively to $\partial M$.}
\end{exm}
\begin{proof}
First case: $M$
is  the annulus or the M\"obius strip. Then,
 the $1$-skeleton of the cellulation dual to $X$
contains an essential simple loop $\ell$ such that
 $M$ cut along $\ell$ is an annulus or two annuli. So, the union $Y\subset X$
 of the cells of $X$ disjoint from $\ell$ collapses onto $\partial M$.
 The pair $(X,Y)$
is matchable (Remark \ref{cycle_trv}), the pair $(Y,\partial M)$ is matchable (Remark \ref{collapse_trv}), and  $X\vert\partial M$
is matchable (Remark \ref{cycle_trv}).

Second case: $M$ is
the $2$-torus or the Klein bottle. Then, the $1$-skeleton of the cellulation dual to $X$
contains an essential simple loop $\ell$ such that $M$ cut along $\ell$ is an annulus.
 Consider the union $Y\subset X$
 of the cells of $X$ disjoint from $\ell$. The pair $(X,Y)$
is matchable (Remark \ref{cycle_trv}) and the annulus $Y$
is matchable (first case).
\end{proof}

Next, recall
Hall's so-called ``marriage theorem''. Let $\Sigma:=\Sigma_0\sqcup \Sigma_1$ be
 a finite, $\Z/2\Z$-graded set and let
 $I\subset\Sigma\times\Sigma$ be a symmetric relation in $\Sigma$, of degree $1$.
 For every subset $A\subset \Sigma$,
 consider the subset $I(A)\subset \Sigma$ of the elements $I$-related
 to at least one element of $A$.
 A \emph{matching} on $\Sigma$ with respect to
  $I$ is a partition of $\Sigma$ into $I$-related pairs.

\begin{thm}[Hall \cite{hall_35}]\label{wedding_pro}
 The following properties are equivalent:
\begin{enumerate}
\item The relation $I$ is matchable;
\item One has $\vert A\vert\le\vert I(A)\vert$ for every $A\subset \Sigma$~;
\item $\vert\Sigma_1\vert=\vert\Sigma_0\vert$ and one has $\vert A\vert\le\vert I(A)\vert$ for every $A\subset\Sigma_0$.
\end{enumerate}
\end{thm}
Also recall that the Ford-Fulkerson 
 algorithm \cite{ford_fulkerson_56}\cite{edmonds_karp_72} computes 
a matching,
 if any, in time $O(\vert\Sigma\vert^2\vert I\vert)$,
 thus giving some (moderate) effectiveness to our existence
results. 

\medbreak
Coming back to polyhedral complexes,
 some examples of unmatchable complexes
 will follow from the \emph{trivial} sense of Hall's criterion.

\begin{exm}
 \emph{A connected simplicial $2$-complex whose Euler characteristic vanishes,
 unmatchable as well as every subdivision.}

Let
$$X:=S^2\vee S^1\vee S^1$$ be the bouquet, at some common vertex $v$,
 of a triangulated $2$-sphere $S^2$ with
two triangulated circles.
 Then, $\chi(X)=0$, but $X$ does not admit
any matching. Indeed, for $A:=\Sigma_0(S^2,v)$,
 one has $I(A)=\Sigma_1(S^2)$, hence $\vert I(A)\vert=\vert A\vert-1$. The same 
holds for any subdivision of $X$.
\end{exm}
\begin{exm}\label{dim_4_exm}
\emph{Some unmatchable triangulated closed connected orientable
 manifolds, whose Euler characteristic vanishes.}

 Let $n=2k$ be even and at least $4$. Start with a closed orientable $n$-manifold $M$ whose Euler
 characteristic $\chi(M)$ is even, divided into
 two parts $M_1$, $M_2$ by a smooth hypersurface $M_0$. Fix a triangulation $X_0$
 of $M_0$.
 
 Recall that by Poincar\'e duality, the Euler characteristic of every closed odd-dimensional
 manifold vanishes. Moreover,
 $$\chi(M)+\chi(M_0)=\chi(M_1)+\chi(M_2)$$
 So, $\chi(M_1)$
 and $\chi(M_2)$ share the same parity.
  Also recall that, $n$ being even, the Euler characteristic of a connected sum of two
 $n$-manifolds $V$, $V'$ is
 $$\chi(V\sharp V')=\chi(V)+\chi(V')-2$$
  Since $\chi((\S^1)^n)=0$
 and $\chi((\S^2)^k)=2^k$, after modifying $M$ by some appropriate number of connected sums
 with $(\S^1)^n$ and/or with $(\S^2)^k$ performed on both sides of
  $M_0$, one can give arbitrary values to $\chi(M_1)$ and $\chi(M_2)$ in the same
  parity class as before, without changing
  $M_0$. In particular, one can arrange that 
  \begin{equation}\label{split_eqn}
\chi(M_2)=-\chi(M_1)>\vert\Sigma_0(X_0)\vert
\end{equation}
Finally, following Armstrong
\cite{armstrong_67}, extend $X_0$ to some triangulation $X_1$ of $M_1$
 and to some triangulation $X_2$ of $M_2$, thus obtaining a global triangulation $X$ of $M$.
 
 Clearly, $\chi(M)=0$. 
We claim that $X$ does not admit
any matching.

 Indeed, for $A:=\Sigma_0(X_2,X_0)$ one has obviously
 $$\vert A\vert=\vert\Sigma_0(X_2)\vert-\vert\Sigma_0(X_0)\vert$$
$$I(A)=\Sigma_1(X_2)$$
Now, recall that
$$\chi(M_2)=\vert \Sigma_0(X_2)\vert-\vert \Sigma_1(X_2)\vert$$
Together with the above inequation (\ref{split_eqn}),
 it follows that $\vert I(A)\vert<\vert A\vert$:
the triangulation $X$ is unmatchable.

Note that, after Theorem \ref{main_thm},
$X$ admits a matchable subdivision.

\end{exm}

\begin{lem}[Acyclic pair]\label{acyclic_lem}

 If $H_*(X,Y)=0$, then the pair $(X,Y)$ is matchable.
\end{lem}
Rational coefficients are understood everywhere; one could as well use $\Z/2\Z$,
or any field.
\begin{proof}
This is an application of Hall's criterion in the realm of elementary
algebraic topology.
For $n\ge 0$, consider as usual 
 the chain vector space $C_n(X,Y)$ of basis $\Sigma^n(X,Y)$; the differential
$$\partial_n:C_n(X,Y)\to C_{n-1}(X,Y)$$
and its kernel $Z_n(X,Y)$.

The following filtration of the pair $(X,Y)$ by
 subcomplexes $X_n\subset X$ is classical. Put $X_0:=Y$.
For $n\ge 1$, let $X_n$ be the union
of $Y$ with the $(n-1)$-skeleton of $X$ and with some $n$-cells which span a linear subspace
 complementary to $Z_{n}(X,Y)$
in $C_n(X,Y)$. Since $H_*(X,Y)=0$, it is straightforwardly verified that $H_*(X_n,Y)=0$
for every $n\ge 0$. Then, the long exact sequence for the relative homologies
of the triad $(X_n,X_{n-1},Y)$ yields
 $H_*(X_{n},X_{n-1})=0$ for every $n\ge 1$.

One is thus
 reduced to prove Lemma \ref{acyclic_lem}
 in the case where moreover, the cells lying in $X$ but not in $Y$ are of only two dimensions:
 $$\Sigma(X,Y)=\Sigma^n(X,Y)\sqcup\Sigma^{n-1}(X,Y)$$ for some
 $n\ge 1$. The pair being acyclic, necessarily
 $$\vert\Sigma^n(X,Y)\vert=\vert\Sigma^{n-1}(X,Y)\vert$$
 For every $A\subset\Sigma(X,Y)$, let  $$\langle A\rangle\subset C_*(X,Y)$$
  denote the spanned
 linear subspace. If moreover $A\subset\Sigma^n(X,Y)$,
recall the set $$I(A)\subset\Sigma^{n-1}(X,Y)$$ of the
 cells incident to at least one cell belonging to $A$;
hence $$\partial_n\langle A\rangle\subset\langle I(A)\rangle$$
Since $\partial_n$ is linear and one-to-one:
 $$\vert A\vert=\dim(\langle A\rangle)=\dim(\partial_n\langle A\rangle)\le\vert I(A)\vert$$
After the equivalence of (1) with (3) in the marriage theorem, 
 the pair $(X_n,X_{n-1})$ is matchable. \end{proof}

\begin{cor}[Subdivision]\label{subdivide_cor} Let $(X,Y)$ be a
pair of polyhedral complexes. Assume that $(X,Y)$ is matchable.

Then,
every polyhedral subdivision $(X',Y')$ of $(X,Y)$  is also matchable.
\end{cor}
\begin{proof}
Consider a matching on $(X,Y)$.
For each matched pair $\sigma,\tau\in\Sigma(X,Y)$ with $\tau\subset\sigma$,
consider the union $$\hat\partial\sigma:=\partial\sigma\setminus Int(\tau)$$ of
the other hyperfaces of $\sigma$.
The restriction $$(X'\vert\sigma,X'\vert\hat\partial\sigma)$$
 is a pair of polyhedral complexes which does of course
  not always collapse, but which always
  admits a matching,
by Lemma \ref{acyclic_lem}. Clearly, the collection of all these partial
matchings constitutes a global matching for the pair of complexes $(X',Y')$.
\end{proof}

\begin{proof}[Proof of Theorem \ref{sphere_thm}] 
Let $X$ be a polyhedral cellulation of a rational
homology sphere $M$ of odd dimension $n$. One can assume that $n\ge 3$.
Fix a $(n-1)$-cell $\sigma$ of $X$ and a hyperface $\tau\subset\sigma$.
 Consider in $X$ the union $Y$ of $\tau$ with the cells of $X$ not containing $\tau$.
First,  the pair $(X,Y)$ is matchable (Remark \ref{cycle_trv}).
 Second, $H_*(Y,\partial\sigma)=0$, hence 
 the pair $(Y,\partial\sigma)$ is matchable (Lemma \ref{acyclic_lem}). Third,
 the polyhedral complex $\partial\sigma$, being
 homeomorphic to the $(n-2)$-sphere,
 is matchable by induction on $n$.
\end{proof}

\begin{cor}[Betti number $1$]\label{betti_cor}
 Let $M$ be  a closed connected $3$-manifold whose first Betti number is $1$.

 Then, every
polyhedral cellulation $X$ of $M$ is matchable.
\end{cor}
\begin{proof}
The 1-skeleton of $X$ (resp. of the dual cellulation) contains
a simple
loop $\ell$ (resp. $\ell^*$) generating $H_1(M)$.
 Consider the union $Y\subset X$ of the cells of $X$ disjoint from $\ell^*$.
  The pair $(X,Y)$ and the circle $\ell$
  are both matchable (Remark \ref{cycle_trv}).
 Also, $H_*(Y,\ell)=0$, hence  the pair $(Y,\ell)$
 is matchable (Lemma \ref{acyclic_lem}).
\end{proof}

Now, consider a triangulation $X$ of a compact manifold $M$ of dimension $n\ge 1$ with
smooth boundary $\partial M$ (maybe empty). If a nonsingular
vector field $\nabla$ on $M$ is transverse to every $(n-1)$-simplex of $X$, we say
for short that $\nabla$ is \emph{transverse to $X$.} 
Note that in particular, $\nabla$ is then transverse to $\partial M$; thus, $\partial M$
splits as the disjoint union of $\partial_s(M,\nabla)$,
where $\nabla$ enters $M$, with $\partial_u(M,\nabla)$, where $\nabla$ exits $M$.
\begin{thm}[Transverse nonsingular vector field]\label{flow_thm}
If the nonsingular
vector field $\nabla$ is transverse to the triangulation $X$, then $X$ is matchable
rel. $\partial_u(M,\nabla)$.
\end{thm}

\begin{proof}
Because of the transversality, for every simplex $\sigma\in\Sigma(X)$ of dimension
less than $n$ and
 not contained in ${\partial_u(M,\nabla)}$ (resp. ${\partial_s(M,\nabla)}$),
  there is a unique {\it 
downstream} (resp. \emph{upstream}) $n$-simplex
$d(\sigma)$ (resp. $u(\sigma)$) $\in\Sigma^n(X)$ containing $\sigma$
and such that the vector field $\nabla$ enters $d(\sigma)$ (resp. exits $u(\sigma)$)
at every point of $Int(\sigma):=\sigma\setminus\partial\sigma$. For $\sigma\in\Sigma^n(X)$,
we agree that $d(\sigma):=\sigma$ and $u(\sigma):=\sigma$.

Consider any $n$-simplex $\delta\in\Sigma^n(X)$ and any face $\sigma\subset\delta$
(the case $\sigma=\delta$ is included.)
We call $\sigma$ \emph{stable} (resp. \emph{unstable}) with respect to $\delta$
 if $\sigma$ does not lie in ${\partial_u(M,\nabla)}$ (resp. ${\partial_s(M,\nabla)}$)
 and if $d(\sigma)=\delta$ (resp. $u(\sigma)=\delta$).
Note that
\begin{itemize}
\item 
 Every \emph{hyperface} of $\delta$ is either stable or unstable;
\item $\delta$
has at least one stable hyperface and at least one unstable hyperface
 (for degree reasons);
\item $\sigma$ is stable 
 if and only if every hyperface of $\delta$ containing $\sigma$
is stable. 
\end{itemize}

Next, for each $\delta\in\Sigma^n(X)$, pick arbitrarily a base vertex
$v(\delta)$ \emph{in the intersection $\partial_-\delta$
 of the unstable hyperfaces of $\delta$} (here of course, it is mandatory that
  $\delta$ is a simplex rather than a general convex polytope.) To this choice, there corresponds canonically a matching, as follows.
For every simplex $\sigma\in\Sigma(X,\partial_u(M,\nabla))$,
 define its {mate} $\bar\sigma$
by:

\begin{enumerate}
\item If $v(d(\sigma))\in\sigma$ then $\bar\sigma$ is the hyperface of
$\sigma$ opposed to $v(d(\sigma))$;
\item  If $v(d(\sigma))\notin\sigma$ then $\bar\sigma$ is the join of
$\sigma$ with $v(d(\sigma))$.
\end{enumerate}
These rules do define a matching on the pair $(X,\partial_u(M,\nabla))$:
 the point here is that $\bar\sigma$ is also a stable
face of $d(\sigma)$. Indeed, if not, then $\bar\sigma$ would be contained
in some unstable hyperface $\eta$ of $d(\sigma)$; but in both cases (1) and (2)
above, this would imply that $\sigma$ itself would be contained in $\eta$, a contradiction.
In other words, $d(\bar\sigma)=d(\sigma)$. It is now clear that
 the map $\sigma\mapsto\bar\sigma$ induces locally,
for each $n$-simplex $\delta$, an involution
in the set of the stable faces of $\delta$; and thus globally a matching on $\Sigma(X,
\partial_u(M,\nabla))$.

Note 1 --- It can be suggestive, for $n=2$ and $n=3$ and for each $0\le i\le n-1$,
 to figure out in $\R^n$, endowed
with the parallel vector field $\nabla:=-\partial/\partial x_n$, a linear $n$-simplex $\delta$
in general position with respect to $\nabla$ and such that
$\dim(\partial_-\delta)=i$; to list the stable faces and the unstable faces;
to choose a base vertex $v\in\partial_-\delta$;
and to compute the corresponding matching between the stable faces.

Note 2 --- We feel that the preceding natural construction is of special interest
with respect to Forman's
general question ``Which smooth vector fields can be triangulated?'' (\cite{forman_05}
section 3).

\end{proof}

In particular, the Hall cardinality conditions also constitute some
\emph{combinatorial} necessary conditions for a
triangulation to admit a transverse nonsingular vector field.
For example, in Example \ref{dim_4_exm}, not
only $X$ does not admit any transverse nonsingular vector field
(which is obvious since such a field would be transverse to $M_0$,
in contradiction with $\chi(M_1,M_0)\neq 0$), but this holds also
for every triangulation of $M$
\emph{combinatorially} isomorphic with $X$; and in particular,
for every jiggling of $X$.

\begin{proof}[Proof of Theorem \ref{main_thm}] 
Since $\chi(M,{\partial_0M})=0$, there is on $M$ a nonsingular vector field $\nabla$ transverse to
$\partial M$, which exits $M$ through ${\partial_0M}$, and
which enters $M$ through $\partial M\setminus{\partial_0M}$. Let $X$ be any triangulation of $M$. Then,
 by W. Thurston's famous Jiggling lemma
\cite{thurston_74},
 one has on $M$ a triangulation $X'$ which is combinatorially isomorphic to some
(iterated crystalline) subdivision of $X$, and which is {transverse} to $\nabla$.
By Theorem \ref{flow_thm}, $X'$ is matchable.
\end{proof}

\begin{proof}[Another proof of Theorem \ref{main_thm} in high dimensions]
Finally, we give an alternative construction for Theorem \ref{main_thm};
this construction works in
every dimension, but $3$.
 Note that, by Corollary \ref{subdivide_cor}
and the Hauptvermutung for smooth triangulations,
 it is enough to construct \emph{one}
triangulation of $M$ matchable relatively to $\partial_0M$.

After Asimov \cite{asimov_75}, since  $n\ge 4$ and $\chi(M,
{\partial_0M})=0$, the pair
$(M,{\partial_0M})$ admits a ``round handle decomposition''.
For each $0\le i\le n-1$, the \emph{round handle} 
 of dimension $n$ and index $i$ is
  defined as $$H^n_i:=\S^1\times\D^i\times\D^{n-i-1}$$
 and one puts $$
{\partial_0 H^n_i}:=\S^1\times
\S^{i-1}\times\D^{n-i-1}$$
 (one agrees that $\S^{-1}=\emptyset$).
By a  \emph{round handle decomposition} for $M$,
one means a filtration of $M$ by submanifolds dimension
$n$, with boundaries and corners:
$$\partial_0M\times[0,1]=M_0\subset M_1\subset\dots\subset M_\ell=M$$
such that, for each $1\le k\le\ell$, one obtains $M_k$ by attaching to $M_{k-1}$
a round handle of dimension $n$ and of some index  $0\le i\le n-1$;
 the attachment map is an embedding $$\partial_0 H^n_i\hookrightarrow
 \partial M_{k-1}\setminus(\partial_0 M\times 0)$$

Fix such a decomposition. Then, choose a triangulation
of $M$ for which each handle is a subcomplex. One is thus reduced to
the two cases
\begin{enumerate}
\item $M=\partial_0M\times[0,1]$; or
\item $M=H_i^n$ and $
{\partial_0M}=\partial_0H_i^n$,
for some $0\le i\le n-1$. 
\end{enumerate}

In case (1),
any triangulation of $M$ is matchable rel. $\partial_0M$
(Lemma $\ref{acyclic_lem}$).

 In case (2), one has a deformation retraction
of $\D^i\times\D^{n-i-1}$  onto its subset
$$K(n,i):=(\D^i\times 0)\cup(\S^{i-1}\times\D^{n-i-1})$$
Hence, $H^n_i$ retracts by deformation onto $\S^1\times K(n,i)$.
 We choose a triangulation of $M$ such
that $\S^1\times K(n,i)$ is a union of cells
of the triangulation. Applying Lemma $\ref{acyclic_lem}$ to the pair
$(M,\S^1\times K(n,i))$,
the proof is reduced to the case where $M=\S^1\times\D^i$ and $
{\partial_0M}=\S^1\times\S^{i-1}$. In that case,
let $X$ be any triangulation of $M$.
The $1$-skeleton of the dual subdivision contains a simple loop $\ell$ homologous to the core $\S^1\times 0$.
Consider the union $Y\subset X$ of the cells of $X$ disjoint from $\ell$.
On the one hand, the pair $(X,Y)$ is matchable (Remark \ref{cycle_trv}).
On the other hand, since $H_*(Y,
{\partial_0M})=0$, the pair $(Y,{\partial_0M})$
 is matchable (Lemma \ref{acyclic_lem}). 
 \end{proof}

\bigskip

Aix-Marseille Universit\'{e}, Centrale Marseille,

I2M (UMR 7373 CNRS),

Technopôle de Château-Gombert,

39, rue Fr\'ed\'eric Joliot-Curie,

13453 MARSEILLE Cedex 13,
France

{gael.meigniez@univ-amu.fr}

\end{document}